\documentclass[12pt,a4paper,draft]{amsart}
\usepackage[english]{babel}
\usepackage{a4wide,amssymb,amsmath,xypic,ifthen}
\usepackage[vcentermath,enableskew]{youngtab}

\newcounter{w}
\renewenvironment{enumerate}{\begin{list}{(\arabic{w})}{\usecounter{w}\setlength{\leftmargin}{50pt}}}
{\end{list}}
\newcommand{\qee}{ \hfill\hspace{2pt}$\triangle$}
\newtheorem{thm}{Theorem}[section]
\newtheorem{corol}[thm]{Corollary}

\newtheorem{prop}[thm]{Proposition}

\theoremstyle{remark}
\newtheorem{rema}[thm]{Remark}
 
\newtheorem{exe}[thm]{Example}
 
\newcommand{\iso}{ \ \raise 4pt\hbox{$\sim$} \kern -10pt\hbox{$\to$}\  }
\newcommand{\cO}{{\mathcal O}}
\newcommand{\PP}{{\mathbb P}}
\newcommand{\Ext}{{\rm Ext}}

\newcommand{\cI}{{\mathcal I}}

\newcommand{\M}{{\mathfrak M}}

\newcommand{\F}{{\mathcal F}}
\newcommand{\E}{{\mathcal E}}

\newcommand{\Z}{{\mathbb Z}}

\newcommand{\C}{{\mathbb C}}
\newcommand{\FF}{{\mathbb F}}

\newcommand{\beq}{\begin{equation}}
\newcommand{\eeq}{\end{equation}}
\newcommand{\beqa}{\begin{eqnarray}}
\newcommand{\eeqa}{\end{eqnarray}}

\newcommand{\uno}{\mbox{1\kern-.59em {\rm l}}}

\newcommand{\nn}{\nonumber}
\newcommand{\be}{\begin{equation}}
\newcommand{\ee}{\end{equation}}
\newcommand{\bea}{\begin{eqnarray}}
\newcommand{\eea}{\end{eqnarray}}

\newcommand{\tinyyoung}[1]{\mbox{\tiny\young(#1)}}
\begin{document}
\begin{flushright} SISSA Preprint 55/2008/FM \\ 
{\tt arXiv:0809.0155 [math.AG]}  
\end{flushright} 
\bigskip\bigskip
\title{Instanton counting on Hirzebruch surfaces}
\bigskip
\date{\today}
\thanks{This research was partly supported by the INFN Research Project PI14 ``Nonperturbative dynamics of gauge theory", by   PRIN    ``Geometria delle variet\`a algebriche"
and by an Institutional Partnership Grant of the Humboldt
foundation of Germany.
 \\[5pt] \indent E-mail: {\tt bruzzo@sissa.it}, {\tt poghos@yerphi.am}, {\tt tanzini@sissa.it} }
 \maketitle \thispagestyle{empty}
\begin{center}{\sc Ugo Bruzzo}$^\ddag$,  {\sc Rubik Poghossian}$^\S$ and {\sc Alessandro Tanzini}$^\ddag$ 
\\[10pt]  {\small 
$^\ddag$ Scuola Internazionale Superiore di Studi Avanzati, \\ Via Beirut 2-4, 34013
Trieste, Italia\\ and Istituto Nazionale di Fisica Nucleare, Sezione di Trieste
}
\\[10pt]  {\small 
$^\S$ Yerevan Physics Institute,
Alikhanian Br. st. 2, \\ 0036 Yerevan, Armenia}
\end{center}

\bigskip\bigskip

\begin{abstract} We perform a study of the moduli space of framed torsion free sheaves on Hirzebruch surfaces by using localization techniques. After discussing general properties of this moduli space, we classify its fixed points under the appropriate toric action and compute its  Poincar\'e polynomial.   From the physical viewpoint,  our results provide the partition function of N=4 Vafa-Witten theory on Hirzebruch surfaces, which is relevant in black hole entropy counting problems according to a conjecture due to Ooguri, Strominger and Vafa.  \end{abstract}

\bigskip
\section{Introduction}

In this paper we study the moduli space of framed torsion-free sheaves on a Hirzebruch surface
$\FF_p$. After providing some clues on the structure of this moduli space, we study a natural toric action on it,
determining its fixed points and the weight decomposition of the tangent spaces at the fixed points. This in turn allows us to compute the Morse indexes at the fixed points, and therefore the Poincar\'e polynomial of the moduli space. 

In this introduction we want to provide some physical motivations and perspectives for this study.
The Ooguri Strominger and Vafa conjecture \cite{osv} relates the counting of microstates of
supersymmetric black holes to that of bound states of Dp-branes wrapping cycles of
the internal Calabi-Yau three-fold in string theory compactifications.
For the D4-D2-D0 system, with the Euclidean D4 branes wrapping a four cycle $M$ of the
Calabi-Yau, this counting problem  has 
a rigorous mathematical definition in terms of the generating functional
of the Euler characteristics of instanton moduli spaces on $M$
with second and first Chern characters fixed by 
to the numer of D0-branes and D2-branes respectively.
In physical terms, this corresponds to the partition function
of the Vafa-Witten twisted N=4 supersymmetric theory on $M$ \cite{VW}. 

In \cite{aosv} the example of a local Calabi-Yau modeled as the total
space of a rank two bundle over a Riemann surface $\Sigma_g$
was studied in detail, providing some support for the
conjecture. 
Very little is known about the moduli space of instantons on such
spaces; indeed the authors of \cite{aosv} proceeded by conjecturing
a reduction to the base $\Sigma_g$.
An argument clarifying this reduction was provided in \cite{BonTa}
via path-integral localisation. Actually, a direct calculation of the instanton partition function
can be performed in a limited number of examples, namely
the case of $\cO_{\PP^1}(-1)$,  where the blowup formulas of
\cite{VW,Yoshi}   apply, and the ALE space $\cO_{\PP^1}(-2)$ \cite{naka-ale}.

A conjectural formula for instanton counting on $\cO_{\PP^1}(-p)$ was given in
\cite{fmr} and   \cite{GSST} by using string theory techniques
and   two-dimensional reduction, respectively. 
Hirzebruch surfaces    $\FF_p$ are  projective
compactifications of the   spaces $\cO_{\PP^1}(-p)$, 
and thus the results we obtain in this paper are related to the results mentioned
above and aim to give a rigorous proof of them.
Notice however that since we are working on a compact manifold
we are able to describe only a subset of bundles, namely those
with integers topological numbers.
We conjecture that our present results describe the topological properties
of the moduli space of instantons on $\cO_{\PP^1}(-p)$ which
are framed to the trivial connection at infinity.  
We provide a check of this conjecture for the case $p=2$, in which
we perform an independent computation of the Poincar\'e polynomial
by using the Kronheimer-Nakajima construction on ALE spaces
\cite{KN}. One can conjecture that bundles on $X_p$ with fractional topological invariants
relate to instantons that are framed to flat connections with nontrivial holonomy at the $\cO_{\PP^1}(-p)$  level,
and could be described in terms of orbifold sheaves at the $\FF_p$ level (or, better to say, on some orbifold replacing $\FF_p$).
We hope to report on this conjecture in a further paper.

While we were working on the completion of this paper, we became aware of the preprint \cite{GaLiu},
which has some overlap with the present work.

{\bf Ackowledgements.} We thank Hua-Lian Chang, Rainald Flume, Francesco Fucito, Emanuele Macr\`\i, Dimitri Markushevich, J.~Francisco Morales and Claudio Rava for useful suggestions. The second author acknowledges hospitality and support from SISSA. 

\bigskip\section{Moduli of framed sheaves on Hirzebruch surfaces} 
In this section we brieflt study the structure of the moduli space of framed sheaves on Hirzebruch surfaces. 
We denote by $\FF_p $ the $p$-th Hirzebruch surface
$\FF_p=\PP(\cO_{\PP^1} \oplus\cO_{\PP^1}(-p))$, which is the projective closure
of the total space $X_p$ of the line bundle $\cO_{\PP^1}(-p)$ on $\PP^1$
(a useful reference about Hirzebruch surfaces is \cite{BHPV}).
This may be explicitly described as the divisor in $ \PP^2\times\PP^1$
$$
\FF_p =  \{ ([z_0 : z_1 : z_2], [z : w] \in \PP^2\times\PP^1\mid
   z_1 w^p = z_2 z^p \},
$$
Denoting by $f: \FF_p \to \PP^2$   the projection onto $\PP^2$,
we fix a line in $\PP^2$ not going through $[1:0:0]$ (``line at infinity") $l_\infty$ and denote
$C_\infty=f^{-1}(l_\infty)$. The Picard group of $\FF_p$ is generated
by $C_\infty$ and the fibre $F$ of the projection $\FF_p\to\PP^1$.
One has
$$C_\infty^2=p,\qquad C_\infty\cdot F=1,\qquad F^2=0\,.$$
The canonical divisor $K_{p}$ may be expressed as
$$
K_{p} = - 2 C_\infty + (p-2) F.
$$

The two-dimensional algebraic torus $\C^\ast\times\C^\ast$ acts
 on $\FF_p$  according to
$$
([z_0 : z_1 : z_2], [z : w]) \longrightarrow
([z_0 : t_1^p z_1 : t_2^p z_2], [t_1 z : t_2 w])
$$
The divisor $C=f^{-1}([1:0:0])$ is invariant under this action. Note that one has
$C=C_\infty-pF$ as divisors modulo linear equivalence. 

We shall consider the moduli space $\M^p(r,k,n)$ parametrizing isomorphism classes
of pairs $(\E,\phi)$, where 
\begin{itemize}\item $\E$ is a torsion-free coherent sheaf on $\FF_p$, whose
topological invariants are the rank $r$, the first Chern class $c_1(\E)=kC$, and the
discriminant
$$ \Delta(\E)=c_2(\E)-\frac{r-1}{2r}c_1^2(\E)= n;$$
\item $\phi$ is a framing on $C_\infty$, , i.e., an isomorphism of the restriction of $\E$ to
$C_\infty$ with the trivial rank $r$ sheaf on $C_\infty$:
$$\phi\colon \E_{\vert C_\infty} \iso \cO_{C_\infty}^{\oplus r}.$$
\end{itemize}
In constructing the moduli space $\M^p(r,k,n)$ one only considers isomorphisms which preserve
the framing, i.e., $(\E,\phi)\simeq(\E',\phi')$ if there is an isomorphism $\psi\colon\E\to\E'$ such that
$\phi'\circ\psi=\phi$. While a point in $\M^p(r,k,n)$  should always be denoted as the class
of a pair $(\E,\phi)$, we shall be occasionally a little sloppy in our notation,  omitting the framing $\phi$.

For $r=1$, the value of $k$ is irrelevant, and every moduli space  $\M^p(1,k,n)$
turns out to be isomorphic to the Hilbert scheme $X_p^{[n]}$ parametrizing
length $n$ 0-dimensional subschemes of $X_p$. The structure of the moduli space
of $r>1$ can be studied using an ADHM description \cite{Rava}, which shows that the moduli space is a smooth algebraic variety of dimension $2rn$. However we are not going
to give this description here.

\subsection{Structure of the moduli space} To study the structure of the   space $\M^p(r,k,n)$ we resort to the theory of \emph{stable pairs}. In this connection see also \cite{Santos}.
\begin{prop} The moduli space $\M^p(r,k,n)$, when nonempty, is a smooth irreducible quasi-projective variety
of dimension $2rn$. Its tangent space at a point $[\E]$ is isomorphic to the vector space
$\Ext^1(\E,\E(-C_\infty))$. 
\end{prop}
\begin{proof} We only give a sketch of the proof here, details will be provided elsewhere. There is a notion of stability for pairs $(\E,f)$ (due to Huybrechts and Lehn \cite{HL1,HL2}), where $\E$ is a torsion-free sheaf on a projective variety
$X$, and $f$ is a morphism $\E\to\F$, where $\F$ is a sheaf supported on a divisor $D\subset X$.
This notion depends on the choice of a polynomial. Taking $X=\FF_p$ and
$D=C_\infty$, one can choose this polynomial so that every framed sheaf on $\FF_p$ is stable in this sense. As a consequence of the theory developed in \cite{HL1,HL2} (that is, by applying Geometric Invariant Theory), the moduli space of stable pairs is  a quasi-projective variety, which is
\emph{fine}, i.e., it has a universal sheaf (see below). The moduli space of framed sheaves
 $\M^p(r,k,n)$ is an open dense subset of the space of stable pairs, and therefore is an irreducible quasi-projective variety (one gets irreducibility by suitably choosing the above-mentioned polynomial). Moreover, the restriction of the universal sheaf on the moduli space of stable pairs provides a universal sheaf on  $\M^p(r,k,n)$.
As we shall later, this also allows one to identity the tangent bundle to the moduli space,
and in particular, provides an isomorphism  $T_{[\E]} \M^p(r,k,n)\simeq \Ext^1(\E,\E(-C_\infty))$.

The statements about the smoothness of the moduli space and its dimension follow from
$$ \Ext^0(\E,\E(-C_\infty)) = \Ext^2(\E,\E(-C_\infty)) =0 .
$$
Indeed, once this is known, the Rieman-Roch theorem implies that
$\dim \Ext^1(\E,\E(-C_\infty))=2rn$. By Serre duality, the $\Ext^2$ group may be written
as $\Ext^0(\E,\E(-C_\infty+(p-2)F))^\ast$. The vanishing of the two groups can be now
be proved as in Proposition 3.2.1 in \cite{King}.  
\end{proof}

\subsection{Universal sheaf and tangent bundle to the moduli space} Again by the general theory developed in \cite{HL1,HL2}, the moduli space $\M^p(r,k,n)$ is \emph{fine}. This means in particular that there is on  the product $\M^p(r,k,n) \times \FF_p$ a torsion-free 
 sheaf $\mathcal U$ such that for every point $[\E]\in \M^p(r,k,n)$ one has
$\mathcal U_{\vert [\E] \times \FF_p} \simeq \E$. By using the abstract Kodaira-Spencer theory \cite{HL-book,HL1}, this allows one to describe the tangent bundle $T\M^p(r,k,n)$ in the form 
\begin{equation}\label{tangent} T\M^p(r,k,n) \simeq {\mathcal E}xt^1_{\pi_1}(\mathcal U,\mathcal U \xrightarrow{\Phi} \pi_2^\ast \cO_{C_\infty}^{\oplus r}) \end{equation}
where $\pi_1$, $\pi_2$ are the projections of $\M^p(r,k,n) \times \FF_p$ onto its factors, 
and $\Phi$ is the framing of the universal sheaf. Moreover, $ {\mathcal E}xt^1_{\pi_1}(\mathcal U,\cdot)$ is the first right derived functor of the functor ${\pi_1}_\ast\circ{\mathcal H}om(\mathcal U,\cdot)$. By base change, and suitable vanishing theorems, the isomorphism \eqref{tangent} implies that at any  $[\E]\in \M^p(r,k,n)$ one has $T_{[\E]} \M^p(r,k,n)\simeq \Ext^1(\E,\E(-C_\infty))$.

 \subsection{Toric action} For $t_1,t_2\in \C^\ast\times\C^\ast$ we have an action
 $G_{t_1,t_2}\colon\FF_p\to\FF_p$ given by $G_{t_1,t_2}([z_0:z_1:z_2])=[z_0:t_1z_1:t_2z_2]$. 
 This action has four fixed points, i.e., $p_1=([1:0:0],[0:1]$ and $p_2=([1:0:0],[1:0]$
 lying on the exceptional line $C$, and two points lying on the line at infinity $C_\infty$. 
 Morever, both $C$ and $C_\infty$ are fixed under this action. The invariance of $C_\infty$ implies that 
the pullback  $G_{t_1,t_2}^\ast$ defines an action on  $\M^p(r,k,n)$. Moreover we have an action of the diagonal maximal torus of
 $Gl(r,\C)$ on the framing. Altogether we have an action of the torus $T=(\C^\ast)^{r+2}$ on 
 $\M^p(r,k,n)$ given by 
\begin{equation}\label{action}
    (t_1,t_2,e_1,\dots,e_r)\cdot (\E,\phi)
    = \left((G_{t_1,t_2}^{-1})^\ast \E, \phi'\right),
\end{equation}
where $\phi'$ is defined as the composition 
$$   (G_{t_1,t_2}^{-1})^* \E_{\vert C_\infty} 
   \xrightarrow{(G_{t_1,t_2}^{-1})^*\phi}
   (G_{t_1,t_2}^{-1})^* \cO_{C_\infty}^{\oplus r}
   \longrightarrow \cO_{C_\infty}^{\oplus r}
   \xrightarrow{e_1,\dots, e_r} \cO_{C_\infty}^{\oplus r}.
$$
 
 We study now the fixed point sets for the action of $T$ on $\M^p(r,k,n)$. This is basically the same statement as in \cite{NY-I} (see also \cite{NY-L} and \cite{GaLiu}),  but for the sake of completeness we give here a sketch of the proof .
  
 \begin{prop} The fixed points of the action \eqref{action} of $T$ on $\M^p(r,k,n)$ are
 sheaves of the type
\begin{equation}\label{fixpo} \E = \bigoplus_{\alpha =1}^r \cI_{\alpha } (k_\alpha C)\end{equation}
 where $\cI_{\alpha }$ is the ideal sheaf of a 0-cycle $Z_\alpha $ supported on $\{p_1\}\cup\{p_2\}$
 and $k_1,\dots,k_r$ are integers which sum up to $k$. The $\alpha$-th factor in this decomposition corresponds, via the framing, to the $\alpha$-th factor in $\cO_{C_\infty}^{\oplus r}$. Moreover, 
\begin{equation}\label{nfixpo} \  n = \ell +\frac{p}{2r}\left(r\sum_{\alpha =1}^rk_\alpha ^2-k^2\right)
= \ell  + \frac{p}{2r} \sum_{\alpha <\beta} (k_\alpha  -k_\beta)^2\,
\end{equation}
 where $\ell$ is the length of the singularity set of $\E$. 
 \end{prop}
 \begin{proof} We first check that if $\E$ is fixed under the $T$-action, then
 $\E=\oplus_{\alpha =1}^r\E_\alpha $, where each $\E_\alpha $ is a $T$-invariant rank-one torsion-free sheaf
 on $\FF_p$. Let $\mathcal K$ be the sheaf of rational functions on $\FF_p$. Then $\E'=\E\otimes\mathcal K$ is free as a $\mathcal K$-module. Choose a trivialization $\E'\simeq \oplus_{\alpha =1}^r\E_\alpha '$ which, when restricted to $C_\infty$, provides an eigenspace decomposition for the action of $T$.
 Then set $\E_\alpha =\E_\alpha '\cap\E$, i.e., pick up the holomorphic sections of $\E'_\alpha $. This provides the desired decomposition. 
 
By taking the double dual of $\E$, and using $T$-invariance, we obtain  $ \E^{\ast\ast} \simeq \bigoplus_{\alpha =1}^r \cO_{\FF_p} (k_\alpha C)$, whence $\E$ has the form \eqref{fixpo}. Since the 0-cycles $Z_\alpha$ have to be $T$-invariant and should not be supported on $C_\infty$, they must be supported as claimed. The numerical equality \eqref{nfixpo} follows from a straight calculation.
 \end{proof}
 
 The exact identification of the fixed points is obtained by using some Young tableaux combinatorics. The precursor of this technique is Nakajima's seminal book \cite{NakaBook} where the case of the Hilbert scheme of points of $\C^2$ is treated (from a gauge-theoretic viewpoint, this is the ``rank 1 case'' for framed instantons on $S^4$). This was generalized to higher rank in \cite{Nek,Flupo,BFMT}. 
 
 A word on notation: if $Y$ is a  Young tableau, $\vert Y \vert $ will denote the   number of boxes in it.  In the case at hand, it turns out that to each fixed point one should attach
 an $r$-ple $\{Y^{(i)}_\alpha\}$ of pairs of Young tableaux (so $i=1,2$ and $\alpha=1,\dots,r$).
 Write $Z_\alpha=Z_\alpha^{(1)}\cup Z_\alpha^{(2)}$, where $Z_\alpha^{(i)}$ is supported at $p_i$.
 The Young tableau  $\{Y^{(i)}_\alpha\}$, for $i$ and $\alpha$ fixed, is attached to the ideal sheaf
 $\cI_{Z_\alpha^{(i)}}$ as usual: choose local affine coordinates $(x,y)$ around $p_i$
 and make a correspondence between the boxes of $\{Y^{(i)}_\alpha\}$ and monomials in $x$, $y$
 in the usual way (cf.~\cite{NakaBook}); then $\cI_{Z_\alpha^{(i)}}$ is generated by the monomials lying outside the tableau. Now the identity \eqref{nfixpo} may be written as 
\begin{equation}\label{countboxes}
n =  \sum_\alpha  \left( \vert
Y_\alpha ^1 \vert +\vert Y_\alpha ^2 \vert \right) +
\frac{p}{2r} \sum_{\alpha <\beta} (k_\alpha  -k_\beta)^2\,.\end{equation}
Looking for all collections of Young tableaux and strings of integers $k_1,\dots,k_r$ satisfying this condition together with $\sum_{\alpha=1}^rk_\alpha=k$, one enumerates all
the fixed points. 

The value of $k$ may be normalized in the range $ 0 \le k < r-1$ upon twisting
by $\cO_{\FF_p}(C)$. In the next Corollary we assume that this has been done.
\begin{corol}  The moduli space $\M^p(r,k,n)$ is nonempty if and only if the number 
\begin{equation}\label{c2} n-\frac{r-1}{2r}pk^2\end{equation} is an integer, and the bound
\begin{equation}\label{bound} n \ge \frac{pk}{2r}(r-k)\,. \end{equation}  holds. \end{corol}
\begin{proof} Assume that \eqref{c2} and \eqref{bound} are satisfied, and let $\ell=n - \frac{pk}{2r}(r-k)$, which is a nonnegative integer. Let $Z$ be a 0-cycle of length $\ell$ whose support does not meet the line $C_\infty$. Then $\E=\cI_{Z}(kC)\oplus\cO_{\FF_p}^{\oplus(r-1)}$ is a sheaf in 
$\M^p(r,k,n)$.

Conversely, if $\M^p(r,k,n)$ is nonempty, then \eqref{c2} is satisfied. Assume that  the bound \eqref{bound} is not satisfied. Since $ \frac{pk}{2r}(r-k)$ bounds from below the quantity 
$\frac{p}{2r} \sum_{\alpha <\beta} (k_\alpha  -k_\beta)^2$ appearing in Eq.~\eqref{countboxes},
the action of $T$ on $\M^p(r,k,n)$ has no fixed points. In view of the calculation of the Poincar\'e polynomial of $\M^p(r,k,n)$ in terms of the fixed points of the toric action (that we shall perform in the next section, independently of this result!), this implies that all Betti numbers of $\M^p(r,k,n)$ are zero. But this is absurd. 
\end{proof}

\bigskip\section{Poincar\'e polynomial}
In this section we determine the weight decomposition of the toric action on the tangent space to the moduli space at the fixed points,
and use this information to compute the Poincar\'e polynomial of the moduli spaces $\M^p(r,k,n)$.

\subsection{Weight decomposition of the tangent spaces} 
We compute here the weights of the action of the torus $T$ on the irreducible subspaces of the tangent spaces to $\M^p(r,k,n)$ at the fixed points $(\E,\phi)$ of the action. According to the decomposition \eqref{fixpo},  the tangent space $T_{(E,\phi)} \M^p(r,k,n) \simeq \Ext^1(\E,
\E(-C_\infty))$ splits as 
$$
   \Ext^1(\E, \E(-C_\infty))
   = \bigoplus_{\alpha ,\beta}
        \Ext^1(\cI_\alpha (k_\alpha  C), \cI_\beta(k_\beta C - C_\infty)).
$$
The factor $\Ext^1(\cI_\alpha (k_\alpha  C), \cI_\beta(k_\beta C - C_\infty))$
has weight $e_\beta e_\alpha ^{-1}$ under the maximal torus of $Gl(r,\C)$. So we need only
 to describe the weight decomposition with respect to the remaining action of 
 $T^2=\C^\ast\times \C^\ast$. 

From the exact sequence
$
0\to \cI_\alpha  \to \cO \to \cO_{Z_\alpha }\to 0
$
we have in K-theoretic terms
\begin{multline}\label{exts}
\Ext^\bullet(\cI_\alpha (k_\alpha  C),   \cI_\beta(k_\beta C - C_\infty))  =
\Ext^\bullet(\cO(k_\alpha  C), \cO(k_\beta C - C_\infty))  \\  
     -   \Ext^\bullet(\cO(k_\alpha  C), \cO_{Z_\beta}(k_\beta C - C_\infty))     -  
  \Ext^\bullet(\cO_{Z_\alpha }(k_\alpha  C), \cO(k_\beta C - C_\infty))   \\   
   +   \Ext^\bullet(\cO_{Z_\alpha } (k_\alpha  C),
           \cO_{Z_\beta}(k_\beta C - C_\infty)).  
\end{multline}
We should note that 
$$
\Ext^0 (\cO(k_\alpha  C), \cO(k_\beta - C_\infty)) \simeq 
H^0(\FF_p, \cO(- n_{\alpha\beta} C - C_\infty)) = 0 \quad\text{where} \quad n_{\alpha\beta}=k_\alpha  - k_\beta
$$
and analogously
$$
\Ext^2 (\cO(k_\alpha  C), \cO(k_\beta - C_\infty)) \simeq H^0(\FF_p,\cO (( n_{\alpha\beta}-1)C - 2 F)) = 0\,.
\label{ext2-0}
$$
These vanishings are again proved as in Proposition 3.2.1 in \cite{King}.
We are thus left with the computation of the $\Ext^1$ groups
$$
\Ext^1(\cO(k_\alpha  C),\cO(k_\beta C - C_\infty)) =
H^1(\FF_p, \cO(-n_{\alpha\beta} C - C_\infty))
$$
We distinguish three cases according to the values of $n_{\alpha\beta}$. In the first case
 ($n_{\alpha\beta}=0$) one again easily sees that  $H^1(\FF_p, \cO(- C_\infty))=0$.
  
In the second case ($n_{\alpha\beta}>0$) we get 
\be
H^1(\FF_p, \cO(- n_{\alpha\beta}C - C_\infty)) \simeq
\oplus_{d=0}^{n_{\alpha\beta}-1} H^0(\PP^1, \cO_{\PP^1}(pd))
\label{ngt1}
\ee 
We prove this result by induction on $n_{\alpha\beta}>0$ using the  exact sequence
\be
0\to \cO(-(n_{\alpha\beta}+1) C - C_\infty) \to \cO (- n_{\alpha\beta}C - C_\infty) \to
\cO_{C}(-n_{\alpha\beta}C ) \to 0
\label{seqn}
\ee
For $n_{\alpha\beta}=1$ one readily obtains \eqref{ngt1}. On the other hand from
\eqref{seqn} we get
\begin{multline*} 0 \to H^0 (\PP^1, \cO(n_{\alpha\beta}p)) \to H^1(\FF_p, \cO(-(n_{\alpha\beta}+1)C - C_\infty)) \\
\to H^1(\FF_p, \cO(-n_{\alpha\beta}C - C_\infty)) \nn  
  \to H^1(\PP^1, \cO(n_{\alpha\beta}p))   = 0
\end{multline*} 
and by the inductive hypothesis (\ref{ngt1})  
\begin{multline*}
 0 \to H^0 (\PP^1, \cO(n_{\alpha\beta}p)) \to H^1(\FF_p, \cO(-(n_{\alpha\beta}+1)C - C_\infty)) \\
\to \oplus_{d=0}^{n_{\alpha\beta}-1} H^0(\PP^1, \cO_{\PP^1}(pd)) \to 0
\end{multline*} 
so that \eqref{ngt1} is proved. 
Since $H^0(\PP^1, \cO_{\PP^1}(pd))$ is the space of homogeneous
polynomials  of degree $pd$ in two variables, it equals  $\sum_{i=0}^{pd}
t_1^{-i} t_2^{-pd+i}$ in the representation ring of $T^2$.
Finally, we get the weights of this summand of the tangent space as
\be
L_{\alpha \beta} (t_1,t_2) =
e_\beta e_\alpha ^{-1} \sum_{d=0}^{n-1}\sum_{i=0}^{pd} t_1^{-i} t_2^{-pd+i}
= e_\beta e_\alpha ^{-1}
\sum_{{i,j\ge 0, i+j\equiv 0 \ {\rm mod} p, \ i+j \le p(n-1)}}
      t_1^{-i} t_2^{-j}
\label{Lmag0}
\ee

For the third case ($n_{\alpha\beta}<0$) we have, again from the exact sequence \eqref{seqn},
$$
H^1(\FF_p, \cO(- nC - C_\infty)) =
\oplus_{d=1}^{-n_{\alpha\beta}} H^1(\PP^1, \cO_{\PP^1}(-pd))\,.
$$
Therefore in this case we have
\be
L_{\alpha \beta} (t_1,t_2) =e_\beta\,e^{-1}_\alpha
\sum_{{i,j\ge 0,\ i+j\equiv 0 \ {\rm mod} p, \ i+j \le - pn - 2}}
      t_1^{i+1} t_2^{j+1}
\label{Lmin0}
\ee

These terms contribute to the weight decomposition of $T_{(E,\phi)} M_p(r,k,n)$
at the fixed points together with  the remaining terms in \eqref{exts}, which are
essentialy the same as in $\PP^2$ case modulo some rescalings of the arguments
(in a sense, we   look at $X_p$ as the resolution of singularities of $\C^2/\Z_p$,
and rescale the arguments to achieve $\Z_p$-equivariance).
In this way we get
\begin{multline} 
   T_{(E,\phi)} \M^p(r,k,n)
   = \\ \sum_{\alpha ,\beta=1}^r \left(L_{\alpha ,\beta}(t_1,t_2)
     + t_1^{p(k_\beta - k_\alpha )} N_{\alpha ,\beta}^{\vec{Y}_1}(t_1^p,t_2/t_1)
     + t_2^{p(k_\beta - k_\alpha )}
     N_{\alpha ,\beta}^{\vec{Y}_2}(t_1/t_2,t_2^p)\right),\label{tangent_space}
\end{multline} 
where $L_{\alpha ,\beta}(t_1,t_2)$ is given by (\ref{Lmag0})
and (\ref{Lmin0}), while
 \be
N_{\alpha ,\beta}^{\vec{Y}}(t_1,t_2)=e_\beta e_\alpha ^{-1}\times
\left\{\sum_{s \in Y_\alpha }
\left(t_1^{-l_{Y_\beta}(s)}t_2^{1+a_{Y_\alpha }(s)}\right)+\sum_{s
\in Y_\beta}
\left(t_1^{1+l_{Y_\alpha }(s)}t_2^{-a_{Y_\beta}(s)}\right)\right\}.
\label{R4_tangent_space} \ee
Here $\vec Y$ denotes an $r$-ple  of Young tableaux, while for a given box $s$ in the tableau
$Y_\alpha$, the symbols $a_{Y_\alpha}$ and $l_{Y_\alpha}$ denote the ``arm" and ``leg" of $s$ respectively, that is, the number of boxes above and on the right to $s$.

\subsection{Topology of the moduli space} We want to compute the Betti numbers of the moduli space
$\M^p(r,k,n)$ using Morse-theoretic arguments. The first step is to replace the noncompact space
$\M^p(r,k,n)$ with a homotopy equivalent compact space  which contains  all fixed points. To this purpose one defines, generalizing \cite{NY-I}, a projective morphism
$ \pi$ from $ \M^p(r,k,n)$ to a space of ``ideal instantons'' on $\PP^2$. We start by noting that
if $[\E]\in \M^p(r,k,n)$, then
\begin{equation*}\label{cherndown} \operatorname{ch}(f_\ast\E) - \operatorname{ch}(R^1f_\ast\E)
= pr + r(1-p) H +  \left[-n -\frac{pk^2}{2r}+\frac{r}2(p-1)+\frac{k}2(2-p)\right] \omega \end{equation*}
where $H$ and $\omega$ are   the hyperplane   and fundamental class in $\PP^2$, respectively.
(Recall that $f$ is the ``blowdown" morphism $\FF_p\to\PP^2$).
Moreover, the double dual $(f_\ast\E)^{\ast\ast}$ is framed to the sheaf
$$f_\ast \cO_{C_\infty} \simeq \cO_{\ell_\infty}^{\oplus r}\oplus \cO_{\ell_\infty}(-1)^{\oplus r(p-1)}$$
on $\ell_\infty$. Next, we let
$$\M_0(pr,N) =  {\mbox{\Large $\amalg$}}_{\ell=0}^N \M_0^{\text{lf}}(pr,N-\ell)\times\operatorname{Sym}^\ell X_p$$
where $$N= n + \frac{pk^2}{2r} +\frac{k}2(p-2)+\frac{r}2(1-p)+\frac{r^2(p-1)^2}{2}\,,$$  
and
 $\M_0^{\text{lf}}(pr,N-\ell)$ is the moduli space of locally free sheaves on $\PP^2$,
of rank $pr$ and second Chern class $N-\ell$, that are framed to the bundle $f_\ast \cO_{C_\infty}$ on $\ell_\infty$.
We define a projective morphism 
$\pi \colon \M^p(r,k,n) \to \M_0(pr,N) $ by letting
$$ (\E,\phi) \mapsto (((f_\ast\E)^{\ast\ast},\phi')), \operatorname{supp}((f_\ast\E)^{\ast\ast}/f_\ast\E)
+ \operatorname{supp}(R^1f_\ast\E) )$$
where $\phi'$ is the framing induced on $(f_\ast\E)^{\ast\ast}$ as noted above. From \eqref{fixpo}  one  notes
that if $\E$ is a fixed point of the toric action, then 
$$(f_\ast\E)^{\ast\ast} \simeq \F = \cO_{\PP^2}^{\oplus r}\oplus \cO_{\PP^2}(-H)^{\oplus r(p-1)}\,.$$
(Here one can note that $\F$ has no deformations as framed sheaf, i.e.,
$\M_0^{\text{lf}}(pr,N-\nu)$ is a point.)
Thus all fixed points map via $\pi$  to the point $(\F,\nu[0]) \in \M_0(r,N)$, where
$$\nu = n+\frac{pk^2}{2r}+\frac{k}2(p-2)\,.$$
Reasoning as in \cite{NY-L}, one can show that $\pi^{-1}((\F,\nu[0]))$ is homotopy equivalent
to  $\M^p(r,k,n) $. Since $\pi^{-1}((\F,\nu[0])$  is compact, and contains all fixed points, we are justified in computing the Poincar\'e polynomial of $\M^p(r,k,n) $ in terms of the indexes of the fixed points of the toric action.

\subsection{Computing the Poincar\'e polynomial}
We shall closely follow the method presented in \cite{NY-L},
part {\bf 3.4}. We choose a one-parameter subgroup $\lambda (t)
$ of the $r+2$ dimensional torus $T$
$$ \lambda (t)=(t^{m_1},
     t^{m_2},t^{n_1},\dots,t^{n_r}),
$$ by specifying generic weights such that
 $$ m_1=m_2\gg n_1>n_2\cdots>n_r>0 \,.$$
 
As explained in \cite{NY-L}, we have now a different fixed locus, but every fixed point has still the form
\eqref{fixpo}, with $Z_\alpha$ fixed under the action of the diagonal subgroup $\Delta$ of $\C^\ast\times\C^\ast$.
Each component of the fixed point set is parametrized by an $r$-ple of pairs
$((k_1,Y_1),\dots,(k_r,Y_r))$ with
$$n=\sum_{\alpha=1}^r\vert Y_\alpha\vert+\sum_{\alpha<\beta}(k_\alpha-k_\beta)^2
\quad\text{and}\quad \sum_{\alpha=1}^r k_\alpha=k\,.$$

To link with the previous construction, we should notice that each fixed point  $\bar\E$  of the
$\Delta\times T^r$ action corresponding to an $r$-ple $((k_1,Y_1),\dots,(k_r,Y_r))$ 
determines a fixed point $((k_1,\emptyset,Y_1),\dots,(k_r,\emptyset,Y_r))$ of the
$T$-action which lies in the same component of the fixed locus as $\bar\E$. Since the $\Delta\times T^r$-module structure of the tangent space $T_{(\E,\phi)}\M^p(r,k,n)$ at a fixed point of $\Delta\times T^r$
does not  change if we move the fixed point in its component, we may choose   $((k_1,\emptyset,Y_1),\dots,(k_r,\emptyset,Y_r))$. In this case
(\ref{tangent_space})
reduces to 
\be
   T_{(E,\phi)} M_p(r,k,n)
   = \sum_{\alpha ,\beta=1}^r \left(L_{\alpha ,\beta}(t_1,t_1)+
   t_1^{p(k_\beta - k_\alpha )} N_{\alpha ,\beta}^{\vec{Y}}(1,t_1^p)\right), \label{reduced_tangent_space}
\ee By (\ref{R4_tangent_space}) we have 
$$
N_{\alpha ,\beta}^{\vec{Y}}(1,t_1^p)=e_\beta e_\alpha ^{-1}\times
\left(\sum_{s \in Y_\alpha } t_1^{p(1+a_{Y_\alpha }(s))}+\sum_{s \in
Y_\beta} t_1^{-p\,a_{Y_\beta}(s)}\right). 
$$
 Our task now is to
compute the index of the critical points, that is,  the number of terms
in (\ref{reduced_tangent_space}) for which one of the following possibilities
holds: 
\begin{enumerate}
\item the weight of $t_1$ is negative,
\item the weight of $t_1$ is zero and the weight of $e_1$ is negative, 
\item the weights of $t_1$, $e_1$ are zero and the weight of
$e_2$ is negative, 
\item the weights of $t_1$, $e_1$, $e_2$ are zero and weight of $e_3$ is
negative,  
\item[ ] ... ...
\item[($r+1$)] the weights of $t_1$, $e_1$,...,$e_{r-1}$ are zero and the weight
of $e_r$ is negative.
\end{enumerate} 
The contribution of the    diagonal ($\alpha =\beta$)
terms of (\ref{reduced_tangent_space})
turns out to be
$$
\sum_{\alpha =1}^r (\mid Y_\alpha \mid -l(Y_\alpha )).
$$
The nondiagonal ($\alpha \neq\beta$)  terms  can be
rewritten as 
\begin{multline}
\label{nondiagonal_tangent_space}
   \sum_{\alpha <\beta}
   \left(\phantom{\Biggl(}L_{\alpha ,\beta}(t_1,t_1)+L_{\beta,\alpha }(t_1,t_1) \right.
   \\ 
   +   \sum_{s \in Y_\alpha }\left(\frac{e_\beta}{e_\alpha }\, t_1^{p(1+a_{Y_\alpha }(s)+k_\beta -
k_\alpha )}+\frac{e_\alpha }{e_\beta}\,t_1^{p\,(-a_{Y_\alpha }(s)-k_\beta
+ k_\alpha )}\right)  
\\  \left.
+\sum_{s \in
Y_\beta}\left(\frac{e_\beta}{e_\alpha }
\,t_1^{p(-a_{Y_\beta}(s)+k_\beta -
k_\alpha )}+\frac{e_\alpha }{e_\beta}\,t_1^{p\,(1+a_{Y_\beta}(s)+k_\alpha 
- k_\beta)}\right)\right),
\end{multline}

We compute now the $L$ terms in these expressions.  Due to the
formulas (\ref{Lmag0}), (\ref{Lmin0}),  if $k_\alpha  -k_\beta \ge 0$
only the term $L_{\alpha ,\beta}(t_1,t_1)$ contributes to the
index. This contribution is easy to count:
$$
 \sum_{{i,j\ge 0,
i+j\equiv 0 \ {\rm mod} p, \ i+j \le
p(n-1)}}1=\sum_{l=0}^{k_\alpha  -k_\beta -1}(l
p+1)=\frac{(k_\alpha -k_\beta)(p\,(k_\alpha -k_\beta-1)+2)}{2}.
$$
If, instead, $k_\alpha  -k_\beta < 0$ there is a contribution from
the $L_{\beta ,\alpha }(t_1,t_1)$ term which is equal to 
$$
\sum_{{i,j\ge 1, i+j\equiv 0 \ {\rm mod} p, \ i+j \le
p(n-1)}}1=\sum_{l=1}^{k_\beta -k_\alpha  -1}(l
p+1)=\frac{(k_\beta-k_\alpha -1)(p\,(k_\beta-k_\alpha )+2)}{2} 
$$
Summarizing, the contribution of the $L$  terms to the index is
\be\label{l_prime}
 l^\prime_{\alpha ,\beta}=\left\{
\begin{array}{ll}\displaystyle\frac{(k_\alpha -k_\beta)(p\,(k_\alpha -k_\beta-1)+2)}{2}
\quad & \text{if}\quad  k_\alpha  - k_\beta\ge 0,
\\[10pt]
\displaystyle \frac{(k_\beta-k_\alpha -1)(p\,(k_\beta-k_\alpha )+2)}{2} \quad &
\text{otherwise.} \end{array}
\right. 
\ee

The only operation to be yet performed is  the sum over the boxes of $Y_\alpha $ in
(\ref{nondiagonal_tangent_space}). A careful analysis shows that the contribution is $\mid
Y_\alpha  \mid +\mid Y_\beta \mid -n^\prime_{\alpha ,\beta}$, where
\be\label{n_prime}
 n^\prime_{\alpha ,\beta}=\left\{ \begin{array}{l} 
\text{$\sharp$ of columns of $Y_\alpha$ that are longer than $k_\alpha  - k_\beta$ if  $k_\alpha  - k_\beta\ge
0$,} \\
\text{$\sharp$  of columns of  $Y_\beta$ that are longer than  
$k_\beta - k_\alpha -1$ otherwise. }\end{array}
\right.
\ee
With all this information we may compute the desired Poincar\'e polynomial.
\begin{thm} The Poincar\'e polynomial of $\M^p(r,k,n)$ is
$$
P_t(\M^p(r,k,n))=\sum_{\text{\rm fixed}\atop\text{\rm points}} \prod_{\alpha =1}^r t^{2(\mid
Y_\alpha \mid -l(Y_\alpha ))} \prod_{i=1}^\infty
\frac{t^{2 (m^{(\alpha)}_i+1)}-1}{t^2-1}
\prod_{\alpha <\beta}t^{2(l^\prime_{\alpha ,\beta}+\mid Y_\alpha 
\mid +\mid Y_\beta \mid -n^\prime_{\alpha ,\beta})}\,.
$$ \label{thm}
\qed \end{thm}
Here  $m_i^{(\alpha)} $ is the number 
of columns in $Y_\alpha$ that have length $i$.

\subsection{Some particular cases} A first, immediate check of Theorem \ref{thm} is that for $r=1$ it should
give the Poincar\'e polynomial of the Hilbert scheme of $n$-point configurations in $X_p$ (and therefore,
of $\PP^1\times \C$). A comparison with the formula in \cite{lothar} tells that this is indeed the case. 
We can also   examine the case of $\M^p(2,0,n)$  in some detail. We set 
$h = k_1=-k_2$ so that $n=p h^2+\sum_{i=1}^\infty i
(m^{(1)}_i+m^{(2)}_i)$. In  view of Eqs.
(\ref{l_prime}) and (\ref{n_prime}) one has
$$
 l^\prime_{1,2}= \left\{
\begin{array}{ll}\text{$h(2+p(2h-1))$} & \text{ if  $h\ge 0$, } \\
\text{$(-2h-1)(-h p+1)$} & \text{if $h< 0$.} \end{array} \right.
$$
$$
n^\prime_{1,2}=\left\{ 
\begin{array}{ll}\text{$ \sum_{i=2h+1}^\infty
m^{(1)}_i$ } & \text{if $h\ge 0$,}
\\
\text{$\sum_{i=-2h}^\infty
m^{(2)}_i$ } & \text{if $h< 0$}. \end{array} \right.  
$$
In this case we may write a generating
function for the Poincare polynomial. One gets 
\begin{multline*}
 \sum_n
P_t(\M^p(2,0,n))q^n=\\ \sum_{k\in \mathbb{Z}}
\sum_{\{m^{(\alpha)}_i\}}\left(q^{pk^2} \prod_{i=1}^\infty
\prod_{\alpha=1,2}q^{im^{(\alpha)}_i}\right)
\times \left( \prod_{i=1}^\infty
\prod_{\alpha=1,2}t^{2(i-1)m^{(\alpha)}_i}\,\,\,
\frac{t^{2(m^{(\alpha)}_i+1)}-1}{t^2-1}\right) \\
 \times \left\{\begin{array}{c} \left(t^{2k(p(2k-1)+2)} \prod_{i=1}^\infty
\prod_{\alpha=1,2}t^{2im^{(\alpha)}_i}\right) \left(
\prod_{i=2k+1}^\infty t^{-2m^{(1)}_1}\right)\quad\text{if}\ k\ge 0 \\
\left(t^{(-2k-1)(-pk+1)} \prod_{i=1}^\infty
\prod_{\alpha=1,2}t^{2im^{(\alpha)}_i}\right) \left(
\prod_{i=-2k}^\infty t^{-2m^{(2)}_1}\right) \quad\text{if}\  k< 0.
\end{array}\right. \end{multline*}
 Performing the summations over
$m^{(\alpha)}_i$ we obtain 
\begin{equation}\label{rk2_poincare}
\begin{split} & \sum_n 
P_t(\M^p(2,0,n))q^n= 
 \\
&  \left(\prod_{i=1}^\infty
\frac{1}{(1-q^i t^{4i})(1-q^it^{4i-2})^2(1-q^it^{4i-4})}\right)  
\\ 
&   \times \left[\sum_{h\ge 0}
   \prod_{i=1}^{2h}\frac{1-q^it^{4i-4}}{1-q^it^{4i}}
\,\,t^{2h(p(2h-1)+2)}q^{ph^2}+
\sum_{h>0}\prod_{i=1}^{2h-1}\frac{1-q^it^{4i-4}}{1-q^it^{4i}}
\,\,t^{2(2h-1)(ph+1)}q^{ph^2} \right]
\end{split}
\end{equation}
 In the case
$p=1$ this result coincides with   Corollary
3.19 of \cite{NY-L}, where it is also shown (identity (3.20)) that  the two terms in square brackets can be further  factorized
  to obtain an elegant product formula. 
  The case $p=2$ is considered in
\cite{sasaki}, where besides usual line bundles $\cO (k_{\alpha}
C)$ with integer $k_\alpha$ the author also considers line bundles with half-integer  Chern class $k_\alpha$ (see
discussion after eq. (3.35) of\cite{sasaki}). It is interesting
that though Eq.~(\ref{rk2_poincare}) in this case can not be
further simplified, incorporating terms corresponding to the
half-integer $k_\alpha$ again leads to a product formula (see
(2.62)-(2.64) of \cite{sasaki}). 
One can speculate that, as we already hinted in the introduction, these
line bundles with fractional Chern classes can be actually interpreted as orbifold bundles.
If one is able to do that, the generating function in 
formula (\ref{rk2_poincare})  could be
fully factorized and expressed in terms of the level $p$ theta
functions.

For  $p=2$ the space  $X_2$ is the $A_1$ ALE space
\cite{KN}, and $\M^2 (r,k,n)$ coincides with the moduli space of $U(r)$ instantons on ALE space $A_1$
with first Chern class $k$ (for details on  instanton counting
on $A_{p-1}$ ALE spaces the reader may refer to \cite{KN,Fucito:2004ry,fmr}). 
In the Appendix we provide evidence that the for $k=0$ the computations performed in the two different 
settings coincide (this has been checked explicitly for contributions involving Young tableaux
with a total number of boxes up to 4).

\section*{Appendix. Comparison with ALE spaces}
We recall from section 3 of \cite{fmr}  the computation of the Poincar\'e
polynomial for the moduli space $M_{A_1}(r,n)$ of  $U(r)$  instantons of zero first Chern class 
and second Chern class $n$ on the ALE space $A_1$, providing evidence that is coincides with the result for the moduli space $\M^2(r,0,n)$.

One should recall that $X_2=A_1$ is the resolution of singularities of the
hyperk\"ahler quotient $\C^2/\Z_2$. To   study the fixed points of the toric action
one can consider the fixed points for the case of $\C^2$ discarding all Young tableaux
that are not $\Z_2$-invariant.
 Fixed points are thus given by $r$ $2$-colored  Young tableaux
$\{ (Y_\alpha, \varepsilon_\alpha) \}$, $\alpha =1,2,\ldots r$,
$\varepsilon_\alpha =0,1$. Each $Y_\alpha$ is a tableau whose   boxes
are filled with $0$ or $1$ as follows: the corner box is
filled with $\varepsilon_\alpha$ and the neighbouring boxes either
in horizontal or in vertical directions are always filled with
different numbers. The weight decomposition around fixed points is
given by: 
$$
T_{(\vec{Y},\,\vec{\varepsilon})}M_{A_1}(r,n)=\sum_{\alpha,
\beta=1}^r
\left(N_{\alpha,\beta}^{\vec{Y}}(t_1,t_2)\right)^{[\Z_2]},
$$
where $N_{\alpha,\beta}^{\vec{Y}}$ is given by
(\ref{tangent_space}) and the superscript $\Z_2$ indicates that
one should pick up only $\Z_2$-invariant terms (i.e. terms which
are invariant under $e_\alpha \rightarrow
(-1)^{\varepsilon_\alpha} e_\alpha$, $t_1\rightarrow -t_1$,
$t_2\rightarrow -t_2$).

Let $N_0$ ($N_1$) be the number of corner boxes filled with $0$
($1$) (of course, $N_0+N_1=r$) and $k_0$ ($k_1$) the total number
of boxes filled with $0$ ($1$). Then the condition $c_1=0$ implies
$N_1+2(k_0-k_1)=0$ and  $n= k_0+N_1/4$. Below we perform explicit
calculation of Poincare polynomials for $r=2$ and $n=1$ or $n=2$.
In both cases $N_0=2$, $N_1=0$ hence $k_0=k_1=1$
in the former and $k_0=k_1=2$ in the latter case.

(i) $k_0=k_1=1$.
Here are the four appropriate fixed points and corresponding
weight decompositions: \bea
\left(\bullet_0,\tinyyoung{1,0}\,\right):\hspace{0.5cm}
\frac{t_1}{t_2}+t_2^2+\frac{t_1t_2e_1}{e_2}+\frac{e_2}{e_1}\,,
\hspace{0.5cm}\text{index}=2; \nonumber \eea \bea
\left(\bullet_0,\tinyyoung{01}\,\right):\hspace{0.5cm} t_1^2+
\frac{t_2}{t_1}+\frac{t_1t_2e_1}{e_2}+\frac{e_2}{e_1}\,,
\hspace{0.5cm}\text{index}=1; \nonumber \eea \bea
\left(\tinyyoung{1,0}\,,\bullet_0\right):\hspace{0.5cm}
\frac{t_1}{t_2}+t_2^2+\frac{t_1t_2e_2}{e_1}+\frac{e_1}{e_2}\,,
\hspace{0.5cm}\text{index}=1; \nonumber \eea \bea
\left(\tinyyoung{01}\,,\bullet_0\right):\hspace{0.5cm} t_1^2+
\frac{t_2}{t_1}+\frac{t_1t_2e_2}{e_1}+\frac{e_1}{e_2}\,,
\hspace{0.5cm}\text{index}=0, \nonumber \eea where to calculate the index
of a critical point we  ordered the generators according to 
$t_2\gg e_1>e_2\gg t_1$. Thus for the Poincar\'e polynomial we get
$$
P_t(M_{A_1}(2,1))=\sum_{\text{fixed}\atop\text{points}}t^{2\,\text{index(fixed\,\,
point)}}=1+2 t^2+t^4,
$$
which is easily seen to match with
(\ref{rk2_poincare}).

(ii) $k_0=k_1=2$.
In this case we have $16$ fixed points: \bea
\left(\bullet_0,\tinyyoung{0101}\,\right):\hspace{0.5cm}
t_1^2+t_1^4+\frac{t_2}{t_1^3}+\frac{t_2}{t_1}+\frac{t_1t_2e_1}{e_2}+
\frac{t_1^3t_2e_1}{e_2}+\frac{e_2}{e_1}+\frac{e_2}{t_1^2e_1}\,,
\hspace{0.5cm}\text{index}=2; \nonumber \eea \bea
\left(\bullet_0,\tinyyoung{1,010}\,\right):\hspace{0.5cm}
t_1^2+\frac{t_1^3}{t_2}+\frac{t_2}{t_1}+\frac{t_2^2}{t_1^2}+\frac{t_1t_2e_1}{e_2}+
\frac{t_1^3t_2e_1}{e_2}+\frac{e_2}{e_1}+\frac{e_2}{t_1^2e_1}\,,
\hspace{0.5cm}\text{index}=3; \nonumber \eea \bea
\left(\bullet_0,\tinyyoung{10,01}\,\right):\hspace{0.5cm}
t_1^2+\frac{t_1}{t_2}+\frac{t_2}{t_1}+t_2^2+\frac{t_1t_2e_1}{e_2}+
\frac{t_1^2t_2^2e_1}{e_2}+\frac{e_2}{e_1}+\frac{e_2}{t_1t_2e_1}\,,
\hspace{0.5cm}\text{index}=3; \nonumber \eea \bea
\left(\tinyyoung{0}\,,\tinyyoung{1,01}\,\right):\hspace{0.5cm}
\frac{t_1^2e_1}{e_2}+\frac{2t_1t_2e_1}{e_2}+\frac{t_2^2e_1}{e_2}+
\frac{2e_2}{e_1}+\frac{t_1e_2}{t_2e_1}+\frac{t_2e_2}{t_1e_1}\,,
\hspace{0.5cm}\text{index}=3; \nonumber \eea \bea
\left(\tinyyoung{01}\,,\tinyyoung{01}\,\right):\hspace{0.5cm}
2t_1^2+\frac{2t_2}{t_1}+\frac{t_1^2e_1}{e_2}+\frac{t_2e_1}{t_1e_2}
+ \frac{t_1^2e_2}{e_1}+\frac{t_2e_2}{t_1e_1}\,,
\hspace{0.5cm}\text{index}=1; \nonumber \eea \bea
\left(\tinyyoung{01}\,,\tinyyoung{1,0}\,\right):\hspace{0.5cm}
t_1^2+\frac{t_1}{t_2}+\frac{t_2}{t_1}+t_2^2+\frac{e_1}{e_2}+
\frac{t_1t_2e_1}{e_2}+\frac{e_2}{e_1}+\frac{t_1t_2e_2}{e_1}\,,
\hspace{0.5cm}\text{index}=2. \nonumber \eea There is no need to do
calculations for the remaining $10$ fixed points since they can be
recovered from those mentioned above via (simultaneous)
transposition of pairs of tableaux (this amounts to replacement
$t_1\leftrightarrow t_2$) or/and exchanging their order (this is
equivalent to the replacement $e_1\leftrightarrow e_2$). Picking
up all the indices of the fixed points for the Poincar\'e polynomial
we get
$$P_t(M_{A_1}(2,2)))=1 + 2t^2 + 5t^4 + 5t^6 +
3 t^8\,,
$$
 which again is what one expects also from
(\ref{rk2_poincare}). In fact, using a Mathematica code we did
similar comparisons also for contributions with $n=3,4$, again finding a perfect
agreement with (\ref{rk2_poincare}).

\par

\end{document}